\documentclass{amsart}

\usepackage{amssymb}

\newtheorem{thm}{Theorem}[section]
\newtheorem{prop}[thm]{Proposition}
\newtheorem{lem}[thm]{Lemma}
\newtheorem{cor}[thm]{Corollary}

\theoremstyle{definition}
\newtheorem{defi}[thm]{Definition}

\theoremstyle{remark}
\newtheorem{rem}[thm]{Remark}

\numberwithin{equation}{section}

\newcommand{\R}{\mathbb{R}^n}

\begin{document}

\title{Unique continuation for fractional orders of elliptic equations}

\author{Hui Yu}
\address{Department of Mathematics, the University of Texas at Austin}
\email{hyu@math.utexas.edu}

\begin{abstract}
We establish the strong unique continuation property of fractional orders of linear elliptic equations with Lipschitz coefficients by establishing almost monotonicity for an Almgren-type frequency functional via an extension procedure.\end{abstract}

 \maketitle

\tableofcontents

\section{Introduction}
It is well-known that analytic functions cannot vanish at infinite order at any point, their vanishing orders being dictated by degrees of their principal polynomials. This implies the strong unique continuation property for Laplace's equation since solutions to this equation are analytic. 

It is then natural to ask whether all elliptic equations impose unique continuation properties like Laplace's. This was first answered in the affirmative for nice equations in two spatial dimensions by Carleman \cite{Ca}. Following his approach, unique continuation properties for more general equations in higher spatial dimensions were studied via various Carleman-type estimates. The reader is encouraged to consult the survey of Tataru \cite{Ta} for more results in this direction. 

In contrast to this `hard' and often technical strategy, Garofalo and Lin discovered a comparatively `soft' geometric-variational approach to unique continuation properties in a series of papers \cite{G}\cite{GL1}\cite{GL2} based on ideas of Aronszajn-Krzywicki-Szarski \cite{AKS}. They proved that, in any spatial dimensions, linear second-order elliptic equations with Lipschitz coefficients enjoy unique continuation properties. This was recently generalized to fully nonlinear uniformly elliptic equations of second order by Armstrong-Silvestre \cite{AS} through a perturbative technique by Savin \cite{Sa}. It is interesting to note that the result of Garofalo-Lin is essentially optimal, considering the counterexamples to higher-order elliptic equations \cite{Co} and to second-order elliptic equations with merely H\"older coefficients \cite{Pl}. 

The goal of this paper is to generalize the result of Garofalo-Lin to elliptic equations of fractional orders between $0$ and $1$. 

The precise definition of a fractional order elliptic equation is given in the next section.  But to get some intuition about these operators, let's consider the following simple case.

Let  $L$ be a uniformly elliptic operator defined on a bounded domain $\Omega\subset\R$. It is well-known that $L$ has discrete spectrum $\{(\lambda_j, e_j)\}$, where $\lambda_j$'s are the eigenvalues, and $e_j$'s are the corresponding eigenfunctions, forming an orthonormal basis of $\mathcal{L}^2(\Omega)$.

Then one can define, for $s\in (0,1)$, the $s$-order version of $L$ by the following $$L^s u=\Sigma \lambda^s_j u_je_j,$$where $u=\Sigma u_je_j$ is the expansion of $u$. This procedure can be done for possibly unbounded domains with more sophisticated spectral theory, and leads to to many interesting operators, including fractional powers of Ornstein-Uhlenbeck operators, Bessel operators, and, more relevant to this work, fractional powers of divergence-type elliptic operators \cite{S}\cite{ST}\cite{SZ}. 

To illustrate the strategy in this paper, it is helpful to see one possible approach  to unique continuation properties of \textit{second order} equations. Here we combine ideas from Garofalo-Lin and the recent work of Fall-Felli \cite{FF}.  

Suppose that $u$ is a nontrivial solution in $B_1\subset \R$ to the equation $$div (A(\cdot)\nabla u)=0.$$ We would like to say that $u$ cannot vanish at infinite order at the origin. It is natural to study the family of blow-ups, as $r\to 0$, $$u_r(x):=u(rx)/r^{\gamma}$$ for  $\gamma>0$.  If this family converge to some nontrivial function, then $u$ has a vanishing order no higher than $\gamma$. For the rigorous definition of vanishing orders, see the next section.

It is then crucial to find the \text{right} $\gamma$.  If $\gamma$ is too small, $\{u_r\}$ converges to the trivial solution. If $\gamma$ is too big, compactness is lost, and there would in general be no limit of $\{u_r\}$.

The idea is to blow up like $$\tilde u_r(x)=\frac{u(rx)}{\sqrt{H(r)/r^{n-1}}},$$ where $H$ is the height function $$H(r):=\int_{\partial B_r} u^2 dH^{n-1}.$$

For this family one has $$H_{\tilde{u}}(1)=\int_{\partial B_1}\tilde{u}_r^2dH^{n-1}=1.$$ Therefore, if they converge, the limit must be nontrivial. On the other hand, compactness of the family follows from \begin{equation}D_{\tilde{u}_r}(1):=\int_{B_1}|\nabla \tilde{u}_r|^2dx\le C.\end{equation} This is a consequence of the almost monotonicity of the so-called frequency functional $$N(r):=rD(r)/H(r):=\frac{r\int_{B_r}\langle A(x)\nabla u,\nabla u\rangle dx}{\int_{\partial B_r}u^2dH^{n-1}}.$$ 

For harmonic functions it is a classical result of Almgren \cite{Al} that $N$ increases with respect to $r$. His proof works for equations of the form $div(\mu(x)\nabla u)=0$ where $\mu$ is a scalar function uniformly bounded away from $0$ and infinity.  What Garofalo and Lin realized in \cite{GL1}\cite{GL2}, based on the work of Aronszajn-Krzywicki-Szarski \cite{AKS}, is that our equation $div(A(x)\nabla u)=0$ transforms into $div_g(\mu(x)\nabla_g u)=0$ once we define the correct Riemannian metric $g$ based on $A$. Using this they established the almost monotonicity of $N$, and thus $N(r)\le CN(1)$ for all $r<1$.

Scaling symmetry then gives $$N_{\tilde{u}_r}(1)=N(r)\le CN(1),$$ and hence $$D_{\tilde{u}_r}(1)\le CN(1)H_{\tilde{u}_r}(1)=CN(1),$$ leading to the compactness (1.1). 

Moreover, being a monotone function, $N(r)$ has a limit as $r\to 0$. And it can be shown that for small $r$, $$H(r)\sim r^{2\lim N(r)}.$$ Thus if we take $\gamma=\lim N(r)$ then $$u_r\sim \tilde{u}_r.$$ The family $\{u_r\}$ converge to a nontrivial function, giving the strong unique continuation property of the equation.

To generalize unique continuation properties to equations of fractional order, one might first look for the right frequency function for solutions. Due to the nonlocal nature of fractional-order equations, it would be very surprising if something like $N$ is monotone, since it is defined in terms of local quantities. 

For the standard fractional Laplacians, however,  this lack of locality can be surpassed with the extension result of Caffarelli-Silvestre \cite{CS}. Given some $s\in (0,1)$, we define the extension of $u:\R\to\mathbb{R}$ to be $U:\R\times(0,\infty)\to \mathbb{R},$ the solution to $$div_{x,y}(y^{1-2s}\nabla_{x,y}U)=0$$ in the upper-half space with boundary datum $u$. Here and in the rest of the paper, $x$ denotes the variable in $n$-dimensional Euclidean space while $y$ denotes the variable in $(0,\infty)$. Depending on the order $s$, the equation for $U$ is either degenerate or singular, but it is entirely \textit{local}. In particular, it makes sense to look at the frequency function of $U$ and \textit{define} it to be the frequency of $u$.

Caffarelli-Silvestre observed $$\Delta^{s}u(x)=\lim_{y\to 0}y^{1-2s}\frac{\partial}{\partial y}U(x,y).$$ Thus if $\Delta^s u=0$ in some domain $\Omega\in\R$, then $\lim_{y\to 0}y^{1-2s}\frac{\partial}{\partial y}U(x,y)=0$ in that domain, which is enough for the monotonicity of the frequency. 

This increasing frequency is used in Fall-Felli \cite{FF} to give the right rate of blow-up for solutions to fractional Laplacians. They gave very precise description of possible blow-up limits in terms of a Neumann eigenvalue problem on the sphere, and in particular this asymptotic analysis shows that nontrivial solutions to $\Delta^s u=0$ have finite vanishing orders. An interesting point is that simply blowing-up in $H^1(\mathbb{R}^{n+1})$ to some nontrivial function is not sufficient. This is because we are blowing up the extension $U$, while the function of interest is $u$, defined on a null subset of $\mathbb{R}^{n+1}$. Thus one has to use blow-up in spaces tighter than $\mathcal{L}^{\infty}$, for instance, H\"older spaces, to catch the pointwise behaviour instead of measure theoretic behaviour of the blow-ups.  This is achieved through compactness of degenerate elliptic equations \cite{FKS}\cite{JLX}.

The work of Fall-Felli has led to a lot of works involving the fractional Laplacian operator, see for instance \cite{Se} and the reference therein. 
To apply their strategy to \textit{variable-coefficient} operators, however, two ingredients are missing. Firstly, we need an extension procedure to get some local equation from our fractional order equation. This would allow us to define the frequency functional for the extended function, corresponding to the extension procedure of Caffarelli-Silvestre.  Fortunately for us, for a large class of operators, this extension procedure has been studied by Stinga and Stinga-Torrea \cite{S} \cite{ST} using spectral theory.

The second ingredient one needs is a monotonicity result for the frequency function, now for a degenerate/ singular variable-coefficient equation. Due to the degeneracy/ singularity near the hyperplane $\{y=0\}$, it seems difficult to find a Riemannian metric as in Garofalo-Lin. For certain degenerate/ singular equations, there have been the works of Tao \cite{T} and Tao-Zhang \cite{TZ}. Our equations are not covered by their results, since our weight $y^{1-2s}$ does not enjoy the assumption of radial symmetry as imposed in their work. However, using the special structure of the extension equation, which results in some cancellations,  we can still prove the almost monotonicity result. Note that as a by-product we have the unique continuation properties for a class of degenerate/ singular elliptic equations, which seems interesting on its own.

To be more precise, we have 

\begin{thm}
Let $A$ be Lipschitz and uniformly elliptic with $A(0)=id$. $L=div(A(\cdot)\nabla)$.

If $u\in H^s(\R)$ satisfies $$L^su=0 \text{ in $B_1$}$$ and $N$ is the frequency function, then the following is nondecreasing $$\bar{N}(r)=e^{Cr}N(r),$$ for some $C$ depending only on elliptic constants, the dimension of the space and the Lipschitz constant of $A$.
\end{thm} 

Here $H^s(\R)$ is the fractional Sobolev space of order $s$. The author would like to thank an anonymous referee for pointing out to the author this characterizaion for domains of fractional orders of elliptic operators. 
The reader is encouraged to consult Caffarelli-Stinga \cite{CSt} and Jin-Li-Xiong \cite{JLX} for more details. 

For the definition of this frequency function $N$, see section 3. 

As an interesting by-product, we establish the strong unique continuation property for elliptic equations with $|y|^{1-2s}$-weights:

\begin{thm}Let $A$ be as in the previous theorem, then 
the equation  $$div(|y|^{1-2s}A(\cdot)\nabla U)=0$$ has the strong unique continuation property.
\end{thm}

\begin{rem}After the completion of this paper, the author learned from Professor H. Koch about the work of R\"uland \cite{Ru}, where she obtained similar results of this type for equations with $C^2$ coefficients in $\R$ via Carleman-type estimates and some ideas of Koch-Tataru \cite{KT}. She also studied  equations involving very rough potentials. \end{rem}

Combining these two previous theorems we can follow the blow-up analysis as in Fall-Felli, and establish the following main result of this work:

\begin{thm}Suppose $s\in (0,1)$, and $L=div(A(\cdot)\nabla)$ is a second-order uniformly elliptic operator with Lipschitz coefficients. 

If $u\in H^s(\R)$ satisfies  $$L^su=0 \text{ in an open $U\subset\R$},$$  and $$u(x_0)=0,$$ then the vanishing order of $u$ at $x_0$ is finite. \end{thm}

Some modification of the argument would give a similar result in bounded domains: 

\begin{thm} Let $\Omega$ be a bounded  Lipschitz domain in $\R$. $s\in (0,1)$. $L=div(A(\cdot)\nabla)$ is a second-order elliptic operator with Lipschitz coefficients. 

If $u\in C^{2,\alpha}_0(\Omega)\cap H^s(\Omega)$ satisfies $$L^su=0 \text{ in an open $U\subset\Omega$},$$ and $$u(x_0)=0,$$ then the vanishing order of $u$ at $x_0$ is finite. 

\end{thm} 

Here $H^s(\Omega)$ is the fractional Sobolev space in domains. 

The definitions of $L^s$  and vanishing orders are given in the next section. $C^{2,\alpha}_0(\Omega)$ is the space of $C^{2,\alpha}(\Omega)$ functions that vanish on $\partial\Omega$.

Compare to the work of Fall-Felli, we do not give blow-up asymptotics around nodal points. Such properties might depend on finer structural conditions on the matrix $A$ to ensure that blow-ups along different sequences would converge to the same limit. 

This paper is organized as follows: In the second section we give some preliminaries about unique continuation properties, and prove some estimates concerning the extension that will be useful for the rest of the paper. We define the Almgren-type frequency function in the next section, and prove its almost monotonicity.  As a byproduct, we obtain the strong unique continuation property for degenerate/ singular elliptic equations with the weight $y^{1-2s}$. Then in the fourth section a blow-up analysis is given around nodal points that establishes Theorem 1.3. In the last section we show how the argument can be modified to give the result concerning bounded domains.

\section{Preliminaries and the extension procedure}
First we give definitions concerning the strong and the weak unique continuation properties \cite{HL}. 
\begin{defi}
For $u\in\mathcal{L}^2(B_1)$ with $u(0)=0$, we define its vanishing order at $0$ to be $$d:=\sup\{k:\limsup_{r\to 0^+} \frac{1}{r^k}(\frac{1}{|B_r|}\int_{B_r}u^2dx)^{1/2}=0\}.$$
\end{defi} 

\begin{defi}
An equation satisfies the \textit{strong unique continuation property} if nontrivial solutions to this equation cannot have infinite vanishing order at any point. 

An equation satisfies the \textit{weak unique continuation property} if nontrivial solutions to this equation cannot vanish in nonempty open sets. 
\end{defi}

\begin{rem}
The strong unique continuation property implies the weak unique continuation property.
\end{rem} 

The following is needed for the extension procedure. For details the reader should consult Stinga-Torrea \cite{ST}.

\begin{defi}
Let $L$ be a densely defined positive linear operator on $\mathcal{L}^{2}(\Omega,d\eta)$ with the spectral resolution $E$. $s\in(0,1)$. 

The heat diffusion semigroup $e^{-tL}: \mathcal{L}^{2}(\Omega,d\eta)\to  \mathcal{L}^{2}(\Omega,d\eta)$ is defined by $$e^{-tL}=\int_0^{\infty} e^{-t\lambda}dE(\lambda).$$

The spectral fractional operator $L^s$, with domain $Dom(L^s)$, is defined by $$L^s=\int \lambda^sdE(\lambda).$$

\end{defi}

The main result in Stinga-Torrea \cite {ST} is 
\begin{thm}For $s\in (0,1)$ and $u\in Dom(L^s)$. 

A solution of the extension problem $$\begin{cases}-L_xU+\frac{1-2s}{y}\frac{\partial}{\partial y}U+\frac{\partial^2}{\partial y^2}U =0 &\text{ in $\Omega\times (0,\infty)$ },\\ U(x,0) =u(x) &\text{ in $\Omega$}\end{cases}$$ is given by  \begin{align*}U(x,y)&=\frac{1}{\Gamma(s)}\int_0^{\infty} e^{-tL}(L^s u)(x)e^{-\frac{y^2}{4t}}\frac{dt}{t^{1-s}}\\&=\frac{y^{2s}}{4^s\Gamma(s)}\int_0^{\infty}e^{-tL}u(x)e^{\frac{-y^2}{4t}}\frac{dt}{t^{1+s}}\end{align*} where $\Gamma$ is the Gamma function.

Moreover, \begin{align*}\lim_{y\to 0^+}\frac{U(x,y)-U(x,0)}{y^{2s}}&=\frac{\Gamma(-s)}{4^s\Gamma(s)}L^su(x)\\&=\frac{1}{2s}\lim_{y\to 0^+}y^{1-2s}\frac{\partial}{\partial y}U(x,y).
\end{align*}
\end{thm} 

It is interesting to see such general result, however, with the exception of the last section of this paper we consider only $\mathcal{L}^{2}(\R)$ equipped with the Lebesgue measure.  Also our linear operator is of the form $u\to -div(A(\cdot)\nabla u)$ or the negative of it, where $A$ is Lipschitz and uniformly elliptic. 

For such operators, we can say much more about their extensions:

\begin{prop}
Let $\Omega$ be either the entire $\R$ or  a Lipschitz domain, and L be of the form $Lu(\cdot)=-div(A(\cdot)\nabla u(\cdot))$ where $A$  is Lipschitz and uniformly elliptic with ellipticity constants $0<\lambda\le\frac{1}{\lambda}<\infty$. 

For $u\in H^s(\Omega)$, the extension given in Theorem 2.5, $U$, is in $H^{1}_{loc}(y^{1-2s}dxdy,\Omega\times (0,\infty))$ with $$\|U\|_{H^1(y^{1-2s}dxdy, \Omega\times (0,R))}\le C \|u\|_{H^s(\Omega)},$$ where $C$ depends on the ellipticity, the spatial dimension, the Lipschitz size of $\Omega$ and the width $R>0$.
 \end{prop} 

Here and in the rest of the paper, for $E\subset\mathbb{R}^{n+1}$, $E^+$ denotes its intersection with $\{y>0\}$, and $E'$ denotes its intersection with the hyperplane $\{y=0\}$.

\begin{proof}
The proof for bounded Lipschitz domains can be found in Caffarelli-Stinga \cite{CSt}. The case for $\Omega=\R$ is in Jin-Li-Xiong \cite{JLX}.\end{proof} 


The last proposition in this section says that when $u$ is the solution to $L^su=0$ in $\Omega'$, we can extend $U$ to the entire strip $\Omega'\times \mathbb{R}$ by an even reflection. This new function would solve an equation which allows us to work with balls instead of half balls. The case where $A=id$ is given in Caffarelli-Silvestre \cite{CS}. For the proof for this variable coefficient case, see Stinga-Zhang \cite{SZ}.

\begin{prop}
Under the same assumptions, and further assume $$L^su=0 \text{ in $B_1'\subset\R$},$$ then the extension to the whole space $$\tilde{U}(x,y)=\begin{cases}U(x,y) &\text{ if $y\ge 0$}\\ U(x,-y) &\text{ if $y< 0$}\end{cases}$$ is a solution to $$div(|y|^{1-2s}A(\cdot)\nabla U)=0 \text{ in $B_1\subset\mathbb{R}^{n+1}$}.$$
\end{prop} 

Here we have used the same letter $A$ to denote the $n$-by-$n$ coefficient matrix as in Theorem 2.6, and the matrix obtained by adding a $1$ as the $(n+1)\times (n+1)$-entry to $A$.
In the sequel we shall use the same letter $U$ to denote the extension given by Theorem 2.5 as well as its even extension to the entire $\mathbb{R}^{n+1}$.

\section{The frequency function}

In this section, we assume that $A(0)=id$,   and that $u\in H^s(\R)$ satisfies $$L^su=0 \text{ in $B_1'$}.$$ We have seen that the extension $U$ solves a local PDE in $B_1\subset\mathbb{R}^{n+1}$. We define the frequency function and prove its almost monotonicity. 

Let $z=(x,y)$ denote points in $\mathbb{R}^{n+1}$. For $z\neq 0$ we define the following $$\mu(z):=\langle A(z)z,z\rangle/|z|^2,  \beta(z):=A(z)z/\mu(z).$$ The choice of these weights is very much inspired by Tao \cite{T} and Tao-Zhang \cite{TZ}, where the authors studied monotonicity properties for degenerate elliptic second order equations. The major difference is that they require radial symmetry of the weight, which we cannot afford. 

We collect a few properties of these two functions:
\begin{lem} For almost every $z\in\mathbb{R}^{n+1}$ and $r>0$,
$$\lambda\le\mu(z)\le\frac{1}{\lambda}.$$

$$|\beta(z)|\le\frac{1}{\lambda}|z|.$$
 $$|\frac{\partial}{\partial r}\mu(rz)|\le Lip_A.$$
$$\frac{\partial\beta_i}{\partial x_j}(z)=\delta_{ij}+O(|z|).$$
\end{lem} 

\begin{rem}
When there is no ambiguity, we shall not distinguish between statements that hold for all points or almost every points.
\end{rem} 
\begin{proof}
The first and the second estimates are direct consequences of the ellipticity of $A$.

The third follows from $\mu(rz)=\langle A(rz)rz,rz\rangle/|rz|^2=\langle A(rz)z,z\rangle/|z|^2$. Consequently $$\frac{\partial}{\partial r}\mu(rz)=\langle \partial_rA(rz)z,z\rangle/|z|^2.$$

For the last estimate, 
$$\frac{\partial}{\partial x_j}\beta_i-\delta_{ij}=(\frac{a_{ij}}{\mu(z)}-\delta_{ij})+\frac{\partial}{\partial x_j}(\frac{a_{ik}}{\mu(z)})x_k.$$ 
Now with \begin{align*}|\frac{a_{ij}}{\mu(z)}-\delta_{ij}|&\le |a_{ij}||\frac{1}{\mu(z)}-1|+|a_{ij}-\delta_{ij}|\\&\le C/\lambda|\langle (A(z)-id)z/|z|,z/|z|\rangle|+|a_{ij}-\delta_{ij}|\\&\le CLip_A|z|.\end{align*} Meanwhile $\frac{\partial}{\partial x_j}(\frac{a_{ik}}{\mu(z)})=\frac{\partial_j(a_{ik})}{\mu(z)}+a_{ik}\partial_j\mu(z)/\mu(z)^2$ is bounded due to the Lipschitz regularity of $A$, the fourth estimate follows.
\end{proof}

For $0<r<1$, the generalized height function is defined to be $$H(r):=\int_{\partial B_r}|y|^{1-2s}\mu(z)U(z)^2dH^{n}(z).$$ The generalized Dirichlet energy is $$D(r):=\int_{B_r}|y|^{1-2s}\langle A(z)\nabla U,\nabla U\rangle dz.$$ 

With the estimates in the previous section, it is easy to see that these functions are well-defined. Moreover, $H(r)=0$ would imply $U=0$ along $\partial B_r$ which would imply $U=0$ by the uniqueness of solution to the degenerate/ singular elliptic equation \cite{FKS}\cite{FKJ}. Consequently, $H(r)$ is nonvanishing for nontrivial $U$, and the following frequency function is well defined $$N(r):=rD(r)/H(r).$$
 
\begin{lem} 
$$H'(r)=(\frac{n+1-2s}{r}+O(1))H(r)+2\int_{\partial B_{r}}|y|^{1-2s}\mu(z)U(z)U_{\nu}(z)dH^n(z),$$ where $U_\nu$ is the exterior normal derivative, and $O(1)$ denotes an error that is bounded in $\mathcal{L}^\infty$ by constants depending only on the Lipschitz norm of $A$. 
\end{lem} 

\begin{proof}To get some ideas, we first give a proof for smooth $U$. 

By a change of variable $$H(r)=r^{n+1-2s}\int_{\partial B_1}|y|^{1-2s}\mu(rz)U^2(rz)dH^n(z).$$ For fixed $r_0>0$, \begin{align*}\lim_{r\to r_0}\frac{r^{n+1-2s}\mu(rz)U^2(rz)-r_0^{n+1-2s}\mu(r_0z)U^2(r_0z)}{r-r_0}=&(n+1-2s)r_0^{n-2s}\mu(r_0z)U^2(r_0z)\\+r_0^{n+1-2s}\frac{\partial}{\partial r}|_{r=r_0}\mu(rz)U^2(r_0z)&+2r_0^{n+1-2s}\mu(r_0z)U(r_0z)U_{\nu}(r_0z), \end{align*} where $\nu$ is the exterior unit normal.

Now for $r$ close to $r_0$, say, $\frac{1}{2}r_0<r<2r_0$,
\begin{align*}
|&\frac{r^{n+1-2s}\mu(rz)U^2(rz)-r_0^{n+1-2s}\mu(r_0z)U^2(r_0z)}{r-r_0}|\\\le & |\frac{r^{n+1-2s}-r_0^{n+1-2s}}{r-r_0}|\mu(rz)U^2(rz)+r_0^{n+1-2s}|\frac{\mu(rz)-\mu(r_0z)}{r-r_0}|U^2(rz)\\&+r_0^{n+1-2s}\mu(r_0z)|\frac{U^2(rz)-U^2(r_0z)}{r-r_0}|\\\le & (n+1-2s)2^{n-2s}r_0^{n-2s}\frac{1}{\lambda}\|U\|^2_{\infty}\\&+r_0^{n+1-2s}Lip_\mu\|U\|^2_{\infty}+r_0^{n+1-2s}\frac{1}{\lambda}\|U\|_{\infty}\sup_{\frac{1}{2}r_0<r<r_0}|U_\nu (rz)|\\\le & (n+1-2s)2^{n-2s}r_0^{n-2s}\frac{1}{\lambda}\|U\|^2_{\infty}\\&+r_0^{n+1-2s}Lip_A\|U\|^2_{\infty}+r_0^{n+1-2s}\frac{1}{\lambda}\|U\|_{\infty}\sup_{\frac{1}{2}r_0<r<r_0}|U_\nu (rz)|.
\end{align*}
For the last inequality we used Lemma 3.1.

Thus  
$|\frac{r^{n+1-2s}\mu(rz)U^2(rz)-r_0^{n+1-2s}\mu(r_0z)U^2(r_0z)}{r-r_0}|$ is uniformly bounded for $r\neq r_0$ and $\frac{1}{2}r_0<r<2r_0$. Hence 
\begin{align*}H'(r_0)&=(n+1-2s)r_0^{n-2s}\int_{\partial B_1}|y|^{1-2s}\mu(r_0z)U^2(r_0z)dH^n(z)\\&+r_0^{n+1-2s}\int_{\partial B_1}|y|^{1-2s}\frac{\partial}{\partial r}|_{r=r_0}\mu(rz)U^2(r_0z)dH^n(z)\\&+2r_0^{n+1-2s}\int_{\partial B_1}|y|^{1-2s}\mu(r_0z)U(r_0z)U_{\nu}(r_0z)dH^n(z)\\&=(\frac{n+1-2s}{r_0}+O(1))H(r_0)+2\int_{\partial B_{r_0}}|y|^{1-2s}\mu(z)U(z)U_{\nu}(z)dH^n(z).\end{align*}

Here $O(1)$ is some function bounded in sup-norm by $Lip_A$.

For general $U\in H^{1}(|y|^{1-2s}dxdy;B_1)$, note that the right-hand side is uniformly bounded on bounded subsets in $H^{1}(|y|^{1-2s}dxdy;B_1)$, the result follows from a standard approximation scheme. 
\end{proof} 

Another expression of $H'$ seems more useful:

\begin{lem}
$$H'(r)=(\frac{n+1-2s}{r}+O(1))H(r)+2D(r).$$
\end{lem} 

\begin{proof}
Using the equation for $U$, one sees $$D(r)=\int_{B_r}div(|y|^{1-2s}UA\nabla U)dz=\int_{\partial B_r}|y|^{1-2s}U\langle A\nu,\nabla U\rangle dH^{n}(z).$$

Now note that along $\partial B_r$, $\nu=z/|z|$, hence $$\langle A\nu-\mu(z)\nu,\nu\rangle=\langle A\nu,\nu\rangle-\mu(z)=\langle Az,z\rangle/|z|^2-\mu(z)=0.$$ That is, $\kappa:=A\nu-\mu(z)\nu$ lives in the tangent bundle of $\partial B_r$, thus one can apply the divergence theorem on $\partial B_r$ to obtain

\begin{align*}D(r)-&\int_{\partial B_r}|y|^{1-2s}\mu(z)UU_\nu dH^n(z)\\&=\int_{\partial B_r}|y|^{1-2s}U\nabla U\cdot\kappa dH^n(z)\\&=-\frac{1}{2}\int_{\partial B_r}|y|^{1-2s}U^2div(\kappa) dH^n(z)-\frac{1}{2}\int_{\partial B_r}U^2\nabla |y|^{1-2s}\cdot\kappa dH^n(z).\end{align*}

Now note that $|div(\kappa)|\le C(n)Lip_A$. 

$\nabla y^{1-2s}\cdot\kappa=(1-2s)(0,0,0,\dots, y^{-2s})\cdot\kappa=(1-2s)y^{-2s}\cdot(1-\frac{1}{\mu(z)})y\le C(s,\lambda)y^{1-2s}.$

Therefore, both terms are of the order $O(1)H(r)$. That is,
$$\int_{\partial B_r}|y|^{1-2s}\mu(z)UU_\nu dH^n(z)=D(r)+O(1)H(r).$$ 

Plug this into the previous Lemma we get the desired estimate.
\end{proof} 

By co-area formula, it is clear that for almost every $r$, $$D'(r)=\int_{\partial B_r}|y|^{1-2s}\langle A\nabla U,\nabla U\rangle dH^n(z).$$ A more useful expression is 
\begin{lem}
$$D'(r)= (\frac{n-2s}{r}+O(1))D(r)+2\int_{\partial B_r}|y|^{1-2s}\frac{1}{\mu} \langle A\nu,\nabla U\rangle^2dH^{n}(z).$$
\end{lem} 

\begin{proof}
We first note the following 
\begin{align*}div(y&^{1-2s}\beta\langle A\nabla U,\nabla U\rangle)-2 div(y^{1-2s}\beta\cdot\nabla U A\nabla U)\\&=div(\beta)y^{1-2s}\langle A\nabla U,\nabla U\rangle+y^{1-2s}\beta_l\frac{\partial a_{jk}}{\partial x_l}\frac{\partial}{\partial x_j}U\frac{\partial}{\partial x_k}U\\&+\beta_{n+1}(1-2s)y^{-2s}\langle A\nabla U,\nabla U\rangle-2y^{1-2s}a_{jk}\frac{\partial\beta_l}{x_k}\frac{\partial}{\partial x_l}U\frac{\partial}{\partial x_j}U.\end{align*}

Integrate over $B_r$, the left-hand side gives \begin{align*}\int_{\partial B_r}|y|^{1-2s}\langle A\nabla U,&\nabla U\rangle\beta\cdot\nu dH^n(z)-2\int_{\partial B_r}|y|^{1-2s}\langle A\nu,\nabla U\rangle\beta\cdot\nabla U dH^n(z)\\&=r\int_{\partial B_r}|y|^{1-2s}\langle A\nabla U,\nabla U\rangle dH^n(z)-2r\int_{\partial B_r}|y|^{1-2s}\frac{\langle A\nu,\nabla U\rangle^2}{\mu(z)} dH^n(z)\\&=rD'(r)-2r\int_{\partial B_r}|y|^{1-2s}\frac{\langle A\nu,\nabla U\rangle^2}{\mu(z)} dH^n(z),\end{align*} where we have used $\beta\cdot\nu=\frac{Ar\nu}\mu\cdot \nu=r$  and $\beta\cdot\nabla U=\frac{1}{\mu(z)}\langle Ar\nu,\nabla U\rangle$ along $\partial B_r$.

Now we evaluate the integral of the right-hand side over $B_r$.

With the fourth estimate in Lemma 3.1, $div(\beta)=n+1+O(r)$, $$\int_{B_r}div(\beta)|y|^{1-2s}\langle A\nabla U,\nabla U\rangle dz=(n+1+O(r))D(r).$$

With the second estimate in Lemma 3.1, $|\beta_l\frac{\partial a_{jk}}{\partial x_l}|=O(r)$. Thus $$\int_{B_r}y^{1-2s}\beta_l\frac{\partial a_{jk}}{\partial x_l}\frac{\partial}{\partial x_j}U\frac{\partial}{\partial x_k}Udz=O(r)D(r).$$

Note that $\beta_{n+1}=\frac{y}{\mu(z)}$, one sees $\beta_{n+1}=y+O(|z|)y$ thus \begin{align*}\int_{B_r}\beta_{n+1}(1-2s)y^{-2s}\langle A\nabla U,\nabla U\rangle dz&=\int_{B_r}(y+O(|z|)y)(1-2s)y^{-2s}\langle A\nabla U,\nabla U\rangle dz\\&=(1-2s)D(r)+O(r)D(r).\end{align*}

For the last term on the right-hand side, one again uses the four estimate in Lemma 3.1 to conclude its integral is of the order $2D(r)+O(r)D(r)$.

Indeed \begin{align*}
\int_{B_r}y^{1-2s}a_{jk}\frac{\partial\beta_l}{x_k}\frac{\partial}{\partial x_l}U\frac{\partial}{\partial x_j}Udz&=\int_{B_r}y^{1-2s}a_{jk}(\delta_{lk}+O(z))\frac{\partial}{\partial x_l}U\frac{\partial}{\partial x_j}Udz\\&=\int_{B_r}y^{1-2s}\langle A\nabla U,\nabla U\rangle dz+O(r)\int_{B_r}y^{1-2s}|\nabla U|^2dz\\&=D(r)+O(r)D(r).
\end{align*}
For the last estimate we have used the ellipticity of $A$.
\end{proof} 

We can now establish the almost monotonicity of the frequency function.

\begin{thm}
Let $A$ be Lipschitz and uniformly elliptic with $A(0)=id$. $L=div(A(\cdot)\nabla)$.

If $u\in H^s(\R)$ satisfies $$L^su=0 \text{ in $B_1'$}$$ then $$\bar{N}(r)=e^{Cr}N(r)$$ is nondecreasing, where $C$ depends only on elliptic constants, the dimension of the space and the Lipschitz constant of $A$.
\end{thm}

\begin{rem}
We have presented the version that is most relevant to our purpose. It is obvious that this almost monotonicity holds for all solutions to $$div(|y|^{1-2s}A(\cdot)\nabla U)=0,$$ that is, this U does not necessarily come from an extension.
\end{rem} 

\begin{proof}
A direct computation gives $N'(r)=N(r)(\frac{1}{r}+D'(r)/D(r)-H'(r)/H(r))$. With previous lemmata,

\begin{align*}\frac{1}{r}&+D'(r)/D(r)-H'(r)/H(r)\\&=\frac{1}{r}+\frac{n-2s}{r}+2\frac{\int_{\partial B_r}|y|^{1-2s}\frac{1}{\mu} \langle A\nu,\nabla U\rangle^2dH^{n}(z)}{D(r)}\\&-\frac{n+1-2s}{r}-2\frac{D(r)}{\int_{\partial B_r}\mu U^2dH^n(z)}+O(1)\\&=2(\frac{\int_{\partial B_r}|y|^{1-2s}\frac{1}{\mu} \langle A\nu,\nabla U\rangle^2dH^{n}(z)}{\int_{\partial B_r}|y|^{1-2s}U\langle A\nu,\nabla U\rangle dH^n(z)}-\frac{\int_{\partial B_r}|y|^{1-2s}U\langle A\nu,\nabla U\rangle dH^n(z)}{\int_{\partial B_r}|y|^{1-2s}\mu U^2dH^n(z)})+O(1).
\end{align*}

With Cauchy-Schwarz,
\begin{align*}&\int_{\partial B_r}|y|^{1-2s}\frac{1}{\mu} \langle A\nu,\nabla U\rangle^2dH^{n}(z)\int_{\partial B_r}|y|^{1-2s}\mu U^2dH^n(z)\\&=\int_{\partial B_r}|y|^{1-2s}(\frac{1}{\mu^{1/2}} \langle A\nu,\nabla U\rangle)^2dH^{n}(z)\int_{\partial B_r}|y|^{1-2s}(\mu^{1/2} U)^2dH^n(z)\\&\ge(\int_{\partial B_r}|y|^{1-2s}\frac{1}{\mu^{1/2}} \langle A\nu,\nabla U\rangle\mu^{1/2} UdH^n(z))^2\\&=(\int_{\partial B_r}|y|^{1-2s}U\langle A\nu,\nabla U\rangle dH^n(z))^2\end{align*}
thus $$N'(r)\ge -CN(r)$$ for some $C$ depending only on elliptic constants, the dimension of the space and the Lipschitz constant of $A$.
\end{proof} 

\begin{rem}
The constant $C$ does not depend on the order of the equation $s$.
\end{rem} 
Here are some natural consequences of this almost monotonicity.

\begin{cor}
If $U$ satisfies in $B_1$ $$div(A(\cdot)\nabla U)=0,$$ then there is a constant $C<\infty$ independent of $0<t<1/2$ such that $$\int_{B_{2t}}|y|^{1-2s}U^2\le C\int_{B_t}|y|^{1-2s} U^2.$$
\end{cor}

\begin{proof}
We follow a classical technique in Garofalo-Lin  \cite{GL1}. See also Han-Lin \cite{HL}.

By Lemma 3.3 and Remark 3.7, one has

$$\frac{d}{dr}\log\frac{H(r)}{r^{n+1-2s}}=2\frac{N(r)}{r}+O(1).$$ Here $O(1)$ is independent of $r$.

Integrate this over the interval $(r,2r)$ and take into consideration that $N(r)\le e^CN(1)$ to obtain $$\log\frac{H(r)}{r^{n+1-2s}}|^{2r}_{r}\le 2e^CN(1)\log r|_{r}^{2r}+Cr.$$
That is,
$H(2r)\le CH(r)$. 

Integrate this for $0<r<t$ to conclude.\end{proof} 

Once we have this doubling property, the unique continuation is a natural consequence:
\begin{proof}(of Theorem 1.2.)
Choose $\ell$ big enough such that $C(\frac{1}{2})^{n+\ell}<1$, where $C$ is the constant in the previous lemma.

For each $m\in\mathbb{N}$, the previous lemma implies
\begin{align*}\int_{B_1}|y|^{1-2s}U^2&\le C^{m}\int_{B_{\frac{1}{2^m}}}|y|^{1-2s}U^2\\&=C^m2^{-mn}\frac{1}{|B_{\frac{1}{2^m}}|}\int_{B_{\frac{1}{2^m}}}|y|^{1-2s}U^2\\&\le C^m(\frac{1}{2})^{m(n+\ell)}(\frac{1}{2})^{-m\ell}\frac{1}{|B_{\frac{1}{2^m}}|}\int_{B_{\frac{1}{2^m}}}|y|^{1-2s}U^2.
\end{align*}

Now suppose $U$ has vanishing order at $0$ greater than $\ell$, then $(\frac{1}{2})^{-m\ell}\frac{1}{|B_{\frac{1}{2^m}}|}\int_{B_{\frac{1}{2^m}}}|y|^{1-2s}U^2$ remains bounded as $\ell\to\infty$, while $C^m(\frac{1}{2})^{m(n+\ell)}$ tends to $0$.

Thus $U=0$ in $B_1$.
\end{proof}

\begin{rem}
As already mentioned in the Introduction, since the function we are studying is $u$, defined on a null set in $\mathbb{R}^{n+1}$. The previous corollary is not enough to conclude $u=0$ if $u$ vanishes at infinite order. We still need to do a blow-up analysis as in Fall-Felli \cite{FF}.
\end{rem} 
\begin{cor}
There exists the limit $\gamma=\lim_{r\to 0^+}N(r)$.
\end{cor}

\begin{proof}
Monotonicity of $\bar{N}$ implies the existence of $\lim_{r\to 0^+}\bar{N}(r)$. Since $\exp(Cr)$ obviously converges to $1$ at $0$, $\lim N(r)=\lim\bar{N}(r)$ also exists.
\end{proof} 

\begin{cor}
Given any $\delta>0$, there exists $C=C(\delta,n,s)>0$ such that $$H(r)\ge Cr^{n+1-2s+2\gamma+\delta}.$$
\end{cor}

\begin{proof}
With Lemma 3.3, one has $H'(r)/H(r)=(\frac{n+1-2s}{r}+O(1))+2N(r)/r$. Integrate over the interval $(r,r_{\delta})$, where $r_{\delta}$ is such that $N(\sigma)<\gamma+\delta$ for all $\sigma<r_{\delta}$, to obtain 

\begin{align*}\int_r^{r_{\delta}}H'(t)/H(t)dt&=\int_r^{r_{\delta}}(\frac{n+1-2s}{t}+O(1))dt+\int_r^{r_{\delta}}2N(t)/tdt\\&\le (n+1-2s)\ln(r_\delta/r)+C(r_\delta-r)+2(\gamma+\delta)\ln(r_\delta/r)\\&\le (n+1-2s+2\gamma+\delta)\ln(r_\delta/r)+C.\end{align*}

That is, $$\ln H(r_\delta)-\ln H(r)\le (n+1-2s+2\gamma+\delta)\ln(r_\delta)-(n+1-2s+2\gamma+\delta)\ln(r)+C $$ hence $H(r)\ge C(n,s,\delta) r^{n+1-2s+2\gamma+\delta}$.
\end{proof}

\section{The blow-up}

In this section we prove the main result via a blow-up analysis. This is motivated by the work of Fall-Felli \cite{FF}, where they gave very detailed description of possible blow-up limits in the case of fractional Laplacian. Since we are dealing with variable-coefficient equations here, it is unlikely that one can obtain such precise information. However, we do have sufficient information to make sure nontrivial solutions cannot vanish at infinite order. 

Without loss of generality we assume $u$ solves the equation in $B_1'$ and $A(0)=id$, so we are in the situation of Section 3.

\begin{proof}(of Theorem 1.3.)

For $1>\tau>0$ define $$U_\tau(z)=\frac{U(\tau z)}{\sqrt{H(\tau)/\tau^{n+1-2s}}}.$$

Their height functions and Dirichlet energies satisfy

\begin{align*}H_{U_\tau}(r):&=\int_{\partial B_r}|y|^{1-2s}\mu(\tau z)U_\tau^2dH^n(z)\\&=H(\tau r)/H(\tau)\end{align*}
and \begin{align*}D_{U_\tau}(r):&=\int_{B_r}|y|^{1-2s}\langle A(\tau z)\nabla U_\tau,\nabla U_\tau\rangle dz\\&=\tau D(\tau r)/H(\tau).
\end{align*}

Therefore $$N_{U_\tau}(r):=rD_{U_\tau}(r)/H_{U_\tau}(r)=\tau rD(\tau r)/H(\tau r)=N(\tau r).$$

Note that in particular, 
\begin{align*}\int_{\partial B_1}|y|^{1-2s}\mu(\tau z)U^2_\tau dH^n(z)&=H_{U_\tau}(1)\\&=H(\tau)/H(\tau)\\&=1.\end{align*}

and \begin{align*}\int_{B_1}|y|^{1-2s}\langle A(\tau z)\nabla U_\tau(z),\nabla U_\tau(z)\rangle dz&=D_{U_\tau}(1)\\&=\tau D(\tau)/H(\tau)\\&=N(\tau)\\&\le e^CN(1)/e^{C\tau}\\&\le CN(1).
\end{align*}

As a result, this family $\{U_\tau\}$ is bounded in $H^1(B_1,|y|^{1-2s}dz)$, and up to a subsequence they converge weakly to some $U_0\in H^1(B_1,|y|^{1-2s}dz)$. 

On the other hand, each $U_\tau$ satisfies $$div(|y|^{1-2s}A(\tau\cdot)\nabla U_\tau)=0 \text{ in $B_1$}.$$ Consequently $$div(|y|^{1-2s}\nabla U_0)=0 \text{ in $B_1$}.$$ Note that this is the equation corresponding to the fractional Laplacian. 

With theory for degenerate elliptic equations with $A_2$ weights \cite{FKS}, we see $\{U_\tau\}$ is also compact in $C^{\alpha_0}_{loc}(B_1)$ for some small $\alpha_0>0$.  A more recent account of interior regularity for operators we are interested in is given in Caffarelli-Stinga \cite{CSt}, where they give a extensive of both the interior regularity and boundary regularity. One could also consult the theory in Jin-Li-Xiong \cite{JLX}. 

Passing to the weak formulation, it is also clear $$\lim_{y\to 0^+}y^{1-2s}\frac{\partial}{\partial y}U_0(x,y)=0 \text{ in $B_1'$}.$$

Therefore Theorem 4.1 of Fall-Felli \cite{FF} applies to $U_0$ and we can conclude 
\begin{equation}u_0(\tau x)/\tau^{\gamma_0}\to |x|^{\gamma_0}\psi(x/|x|) \text{ in $C_{loc}^{1,\alpha'}(B_1')$},\end{equation}where $u_0(x)=U_0(x,0)$, $\gamma_0=\lim_{r\to 0}N_{U_0}(r)$ and $\psi$ is a nontrivial function on the sphere.

Note that $U_\tau\to U_0$ weakly in $H^1(B_1;|y|^{1-2s}dz)$ implies for fixed $r>0$, $\liminf_{\tau} D_{U_{\tau}}(r)\ge D_{U_0}(r)$ and $\lim_{\tau} H_{U_\tau}(r)=H_{U_0}(r)$. Hence $$\gamma_0=\lim_{r} N_{U_0}(r)\le \lim_r \liminf_{\tau} N_{U_\tau}(r)=\lim_r \liminf_{\tau} N(\tau r)=\gamma.$$

Consequently, $$\int_{B_1'} (u_0(\tau x)/\tau^{\gamma})^2dx\ge \int_{B_1'} (u_0(\tau x)/\tau^{\gamma_0})^2dx\ge c>0$$ when $\tau$ is small enough. The last inequality is a consequence of (4.1). 

We fix such a small $\tau>0$.

Now for small $r>0$, \begin{align*}
\int_{B_{r\tau }'}u^2(x)dx&=r^n\int_{B_{\tau}'}u^2(rx)dx\\&=r^n\frac{H(r)}{r^{n+1-2s}}\int_{B'_\tau}\frac{u^2(rx)}{H(r)/r^{n+1-2s}}dx\\&=r^n\frac{H(r)}{r^{n+1-2s}}\int_{B'_\tau}u^2_r(x)dx\\&=r^n\frac{H(r)}{r^{n+1-2s}}\tau^{n+2\gamma}\int_{B_1'}u^2_r(\tau x)/\tau^{2\gamma}dx\\&\ge r^n\frac{H(r)}{r^{n+1-2s}}\tau^{n+2\gamma}\frac{1}{2}c.
\end{align*}For the last inequality we used the uniform convergence of $u_r$ to $u_0$.

Thus for $k>0$ 
$$\frac{1}{(\tau r)^k}(\frac{1}{|B_{r\tau }'|}\int_{B_{r\tau }'}u^2(x)dx)^{1/2}\ge \sqrt{1/2}c^{1/2}\tau^{\gamma-k}\sqrt{H(r)/r^{n+1-2s+2k}}.$$ With Corollary 3.12, the right-hand side goes to infinity whenever $k>\gamma$, thus the vanishing degree of $u$ is no larger than $\gamma=\lim_{r\to 0}N(r)$.
\end{proof}

\section{The case of bounded domains}
For a bounded Lipschitz domain $\Omega\subset\R$, we can still define fractional order of elliptic operators using spectral theory \cite{S}\cite{ST}\cite{SZ}. Moreover the extension procedure works in the same fashion \cite{CSt}. In particular, regularity estimates in Section 2 hold true when $\R$ is replaced by $\Omega$.

If $u$ solves a fractional order equation in $U\subset\Omega$, then we can again reflect $U$ to both sides of $U$ using the same argument. Observe then that for results in Section 3 and 4 one only needs the degenerate/ singular \textit{local} equation for the reflected extension, hence they are true for the case of bounded domains.

To conclude Theorem 1.4 is true. 

\section*{Acknowledgements}
The author would like to thank his PhD advisor Prof. Luis Caffarelli for his constant encouragement and guidance. The author is also grateful to Dennis Kriventsov and Yi-Hsuan Lin for many fruitful discussions and Pablo Ra\'ul Stinga for a careful reading of an earlier version of this paper and his many insightful suggestions. The huge improvement of this paper from its previous version was entirely due to the comments of two anonymous referees. 

Part of this work was done during a visit to Hong Kong University of Science and Technology. The author would like to thank Tianling Jin for the invitation and the staff for their hospitality.



\end{document}